\documentclass[12pt]{amsart}

\usepackage{amsmath,amssymb,amsthm}
\usepackage[all]{xy}
\usepackage[dvipdfm,colorlinks=true]{hyperref}

\numberwithin{equation}{section}

\SelectTips{eu}{12}

\newtheorem{theorem}{Theorem}[section]
\newtheorem{proposition}[theorem]{Proposition}
\newtheorem{lemma}[theorem]{Lemma}
\newtheorem{corollary}[theorem]{Corollary}

\theoremstyle{definition}

\theoremstyle{remark}
\newtheorem{remark}[theorem]{Remark}

\newcommand{\ZZ}{\mathbb{Z}}
\newcommand{\RR}{\mathbb{R}}
\newcommand{\M}{\mathcal{M}}

\title{Homotopy decomposition of diagonal arrangements}

\author{Kouyemon Iriye}
\address{Department of Mathematics and Information Sciences, Osaka
Prefecture University, Sakai, 599-8531, Japan}
\email{kiriye@mi.s.osakafu-u.ac.jp}
\author{Daisuke Kishimoto}
\address{Department of Mathematics, Kyoto University, Kyoto, 606-8502, Japan}
\email{kishi@math.kyoto-u.ac.jp}

\begin{document}

\maketitle

\begin{abstract}
Given a space $X$ and a simplicial complex $K$ with $m$-vertices, the arrangement of partially diagonal subspaces of $X^m$, called the dragonal arrangement, is defined. We decompose the suspension of the diagonal arrangement when $2(\dim K+1)<m$, which generalizes the result of Labassi \cite{L}. As a corollary, we calculate the Euler characteristic of the complement $X^m-\Delta_K(X)$ when $X$ is a closed connected manifold.
\end{abstract}

\baselineskip 16pt

\section{Introduction and statement of results}

A homotopy decomposition is a powerful tool in studying the topology of subspace arrangements and their complements. Ziegler and \u{Z}ivaljevi\'c \cite{ZZ} give a homotopy decomposition of the one point compactification of affine subspace arrangements, from which one can deduce the well known Goresky-MacPherson formula \cite{GM}. Bahri, Bendersky, Cohen, and Gitler \cite{BBCG} give a homotopy decomposition of polyhedral products, a generalization of coordinate subspace arrangements and their complements, after a suspension, from which one can deduce Hochster's formula on related Stanley-Reisner rings. A homotopy decomposition of polyedral products due to Grbi\'c and Theriault \cite{GT} and the authors \cite{IK1,IK2} also implies the Golod property of several related simplicial complexes. In this paper, we consider a homotopy decomposition of diagonal arrangements which is defined as follows. Given a space $X$, we assign a partially diagonal subspace of $X^m$ corresponding to a subset $\sigma\subset[m]=\{1,\ldots,m\}$ as
$$\Delta_\sigma(X)=\{(x_1,\ldots,x_m)\in X^m\,\vert\,x_{i_1}=\cdots=x_{i_k}\text{ for }\{i_1,\ldots,i_k\}=[m]-\sigma\}.$$
Throughout the paper, let $K$ be a simplicial complex on the index set $[m]$, possibly with ghost vertices, where we always assume that the empty subset of $[m]$ is a simplex of $K$. We define the arrangement of partially diagonal subspaces of $X^m$ as
$$\Delta_K(X)=\bigcup_{\sigma\in K}\Delta_\sigma(X),$$
which is called the diagonal arrangement associated with $K$. Since $\Delta_K(X)$ is actually the union of the partially diagonal subspaces $\Delta_F(X)$ for facets $F$ of $K$, it is also called the hypergraph arrangement associated with the hypergraph whose edges are facets of $K$. Diagonal arrangements include many important subspace arrangements. For example, if $K$ is the $(m-3)$-skeleton of $(m-1)$-simplex, $\Delta_K(X)$ is the braid arrangement of $X$. Topology and combinatorics of diagonal arrangements have been studied in several directions. See \cite{Ko,PRW,Ki,KS,L,MW,M} for example. We are particularly interested in the homotopy type of $\Delta_K(X)$. Labassi \cite{L} showed that the suspension $\Sigma\Delta_K(X)$ decomposes into a certain wedge of smash products of copies of $X$ when $K$ is the $(m-d-1)$-skeleton of the $(m-1)$-simplex and $2d>m$, in which case $\Delta_K(X)$ consists of all $(x_1,\ldots,x_m)\in X^m$ such that at least $d$-tuple of $x_i$'s are identical. The proof for this decomposition in \cite{L} heavily depends on the symmetry of the skeleta of simplices, and then it cannot apply to general $K$. The aim of this note is to generalize this result to arbitrary $K$ with $2(\dim K+1)<m$ by a new method, where the result is best possible in the sense that if $2(\dim K+1)\ge m$, the decomposition does not hold as is seen in \cite{L}. 

\begin{theorem}
\label{main}
If $X$ is a connected CW-complex and $2(\dim K+1)<m$, then
$$\Sigma\Delta_K(X)\simeq\Sigma(\bigvee_{\sigma\in K}\widehat{X}^{|\sigma|}\vee\widehat{X}^{|\sigma|+1})$$
where $\widehat{X}^k$ is the smash product of $k$-copies of $X$ for $k>0$ and $\widehat{X}^0$ is a point.
\end{theorem}

As a corollary, we  calculate the Euler characteristic of the complement of the diagonal arrangement $\M_K(X)=X^m-\Delta_K(X)$.

\begin{corollary}
\label{euler-char}
Let $X$ be a closed connected $n$-manifold. If $2(\dim K+1)<m$, the Euler characteristic of $\M_K(X)$ is given by
$$\chi(\M_K(X))=\chi(X)^m-(-1)^{mn}\chi(X)(1+\sum_{\emptyset\ne\sigma\in K}(\chi(X)-1)^{|\sigma|}).$$
\end{corollary}

\begin{remark}
Corollary \ref{euler-char} does not hold without compactness of $X$. For example, if $X=\RR$ (hence $n=1$) and $K$ consists only of the empty subset of $[m]$, $\M_K(X)$ is the off-diagonal subset of $\RR^m$ which has the homotopy type of $S^{m-2}$. Then $\chi(\M_K(X))=1+(-1)^{m}$, which differs from Corollary \ref{euler-char}.
\end{remark}

\subsubsection*{\textsc{Acknowledgement}}

The authors are grateful to Sadok Kallel for introducing the paper \cite{L} to them.

\section{Proofs}

Before considering the proof of Theorem \ref{main}, we prepare two lemmas on homotopy fibrations. 

\begin{lemma}
[{\cite[Proposition, pp.180]{F}}]
\label{hocolim-fibration}
Let $\{F_i\to E_i\to B\}_{i\in I}$ be an $I$-diagram of homotopy fibrations over a fixed connected base $B$. Then 
$$\underset{I}{\mathrm{hocolim}}\,F_i\to\underset{I}{\mathrm{hocolim}}\,E_i\to B$$ 
is a homotopy fibration.
\end{lemma}

\begin{lemma}
\label{decomp}
Consider a homotopy fiberation $F\xrightarrow{j}E\xrightarrow{\pi}B$ of connected CW-complexes. If $\Sigma j:\Sigma F\to\Sigma E$ has a homotopy retraction, then
$$\Sigma E\simeq\Sigma B\vee\Sigma F\vee\Sigma(B\wedge F).$$
\end{lemma}

\begin{proof}
Let $r:\Sigma E\to\Sigma F$ be a homotopy retraction of $\Sigma j$, and let $\rho$ be the composite
$$\Sigma E\to\Sigma E\vee\Sigma E\vee\Sigma E\xrightarrow{\Sigma\pi\vee r\vee\Delta}\Sigma B\vee\Sigma F\vee\Sigma(E\wedge E)\xrightarrow{1\vee 1\vee(\pi\wedge r)}\Sigma\check{B}$$
where $\check{A}=A\vee F\vee(A\wedge F)$ for a space $A$. Since $\Sigma E$ and $\Sigma B\vee\Sigma F\vee\Sigma(B\wedge F)$ are simply connected CW-complexes, it is sufficient to show that $\rho$ is an isomorphism in homology by the J.H.C. Whitehead theorem. We first observe the special case that there is a fiberwise homotopy equivalence $\theta:B\times F\to E$ over $B$. Then it is straightforward to see 
$$\rho_*\circ\theta_*(b\times f)=b\times\hat{\theta}_*(f)+\sum_{|b_i|<|b|}b_i\times f_i$$
for singular chains $b,b_i$ in $B$ and $f,f_i$ in $F$, where we omit writing the suspension isomorphism of homology and $\hat{\theta}$ is a self-homotopy equivalence of $F$ given by the composite
$$\Sigma F\xrightarrow{j}\Sigma(B\times F)\xrightarrow{\theta}\Sigma E\xrightarrow{r}\Sigma F.$$
This readily implies that the map $\rho\circ\theta$ is an isomorphism in homology, and then so is $\rho$. For non-connected $B$, the above is also true if we assume that $r$ is a homotopy retraction of the suspension of the fiber inclusion on each component of $B$. We next consider the general case. Let $B_n$ be the $n$-skeleton of $B$, and let $E_n=\pi^{-1}(B_n)$. We prove that the restriction $\rho\vert_{\Sigma E_n}:\Sigma E_n\to\Sigma\check{B}_n$ is an isomorphism in homology by induction on $n$. Since $B$ is connected, $j$ is homotopic to the composite
$$F\simeq\pi^{-1}(b)\xrightarrow{\text{incl}}E$$
for any $b\in B$. Then $\rho\vert_{\Sigma E_0}:\Sigma E_0\to\Sigma\check{B}_0$ is an isomorphism in homology. Consider the following commutative diagram of homology exact sequences.
\begin{equation}
\label{rho}
\xymatrix{
\cdots\ar[r]&H_*(E_{n-1})\ar[r]\ar[d]^{(\rho\vert_{\Sigma E_{n-1}})_*}&H_*(E_n)\ar[r]\ar[d]^{(\rho\vert_{\Sigma E_n})_*}&H_*(E_n,E_{n-1})\ar[r]\ar[d]^{(\rho\vert_{\Sigma E_n})_*}&\cdots\\
\cdots\ar[r]&H_*(\check{B}_{n-1})\ar[r]&H_*(\check{B}_n)\ar[r]&H_*(\check{B}_n,\check{B}_{n-1})\ar[r]&\cdots}
\end{equation}
By the induction hypothesis, $(\rho\vert_{\Sigma E_{n-1}})_*$ is an isomorphism. Since $B_{n-1}$ is a subcomplex of $B_n$, there is a neighborhood $U\subset B_n$ of $B_{n-1}$ which deforms onto $B_{n-1}$. By the excision isomorphism, there is a commutative diagram of natural isomorphisms
$$\xymatrix{H_*(E_n,E_{n-1})\ar[r]^(.48)\cong\ar[d]^{(\rho\vert_{\Sigma E_n})_*}&H_*(E_n,\pi^{-1}(U))\ar[d]^{(\rho\vert_{\Sigma E_n})_*}&H_*(E_n-E_{n-1},\pi^{-1}(U)-E_{n-1})\ar[d]^{(\rho\vert_{\Sigma E_n})_*}\ar[l]_(.63)\cong\\
H_*(\check{B}_n,\check{B}_{n-1})\ar[r]^(.54)\cong&H_*(\check{B}_n,\check{U})&H_*(\check{B}_n-\check{B}_{n-1},\check{U}-\check{B}_{n-1})\ar[l]^(.63)\cong}$$
where we may chose the basepoints of $B_n$ and $U$ in $U-B_{n-1}$ since $B$ is connected. Since each connected component of $B_n-B_{n-1}$ is contractible, $E_n-E_{n-1}$ is fiberwise homotopy equivalent to $(B_n-B_{n-1})\times F$ over $B_n-B_{n-1}$, and then so is also $\pi^{-1}(U)-E_{n-1}$ to $(U-B_{n-1})\times F$ over $U-B_{n-1}$. As in the 0-skeleton case, we see that $\Sigma r$ restricts to a homotopy retraction of the suspension of the fiber inclusion on each component of $\Sigma(B_n-B_{n-1})$. Then by the above trivial fibration case, we obtain that the map 
$$(\rho\vert_{\Sigma(E_n-E_{n-1})})_*:H_*(E_n-E_{n-1},\pi^{-1}(U)-E_{n-1})\to H_*(\check{B}_n-\check{B}_{n-1},\check{U}-\check{B}_{n-1})$$
is an isomorphism, hence so is the right $(\rho\vert_{\Sigma E_n})_*$ in \eqref{rho}. Thus by the five lemma, the middle $(\rho\vert_{\Sigma E_n})_*$ in \eqref{rho} is an isomorphism. We finally take the colimit to get that the map $\rho$ is an isomorphism in homology as desired, completing the proof.
\end{proof}

\begin{remark}
If we assume further that $F$ is of finite type, it immediately follows from the Leray-Hirsch theorem that the map $\rho$ is an isomorphism in cohomology with any field coefficient, implying that $\rho$ is an isomorphism in the integral homology by \cite[Corollary 3A.7]{H}.
\end{remark}

We now consider the diagonal arrangement $\Delta_K(X)$. Suppose that $2(\dim K+1)<m$, or equivalently, $2|\sigma|<m$ for any $\sigma\in K$. Then for $(x_1,\ldots,x_m)\in\Delta_K(X)$, there is unique $x\in X$ such that $x_{i_1}=\cdots=x_{i_k}=x$ with $i_1<\cdots<i_k$ and $2k>m$. Then by assigning such $x$ to $(x_1,\ldots,x_m)\in\Delta_K(X)$, we get a continuous map 
$$\pi:\Delta_K(X)\to X.$$
For $\tau\subset[m]$, let $X^\tau=\{(x_1,\ldots,x_m)\in X^m\,\vert\,x_i=*\text{ for }i\in[m]-\tau\}$, and we put
$$X^K=\bigcup_{\sigma\in K}X^\sigma$$
which is called the polyhedral product or the generalized moment-angle complex associated with the pair $(X,*)$ and $K$.
Observe that for $2(\dim K+1)<m$, we have $\pi^{-1}(*)=X^K$.

\begin{proposition}
\label{fibration}
If $X$ is a CW-complex and $2(\dim K+1)<m$, then $X^K\to\Delta_K(X)\xrightarrow{\pi}X$ is a homotopy fibration. 
\end{proposition}

\begin{proof}
For each $\sigma\in K$, the map $\pi\vert_\sigma:\Delta_\sigma(X)\to X$ is identified with the projection from the product of copies of $X$. Then it follows from Lemma \ref{hocolim-fibration} that 
$$\underset{K}{\mathrm{hocolim}}\,X^\sigma\to\underset{K}{\mathrm{hocolim}}\,\Delta_\sigma(X)\to X$$
is a homotopy fibration. Since the inclusions $X^\sigma\to X^\tau$ and $\Delta_\sigma(X)\to\Delta_\tau(X)$ for any $\sigma\subset\tau\subset[m]$ are cofibrations, we have
$$\underset{K}{\mathrm{hocolim}}\,X^\sigma\simeq\underset{K}{\mathrm{colim}}\,X^\sigma=X^K\quad\text{and}\quad\underset{K}{\mathrm{hocolim}}\,\Delta_\sigma(X)\simeq\underset{K}{\mathrm{colim}}\,\Delta_\sigma(X)=\Delta_K(X),$$
completing the proof.
\end{proof}

Put $\widehat{X}^K=\bigvee_{\emptyset\ne\sigma\in K}\widehat{X}^{|\sigma|}$. In \cite{BBCG}, it is proved that there is a homotopy equivalence
\begin{equation}
\label{BBCG}
\epsilon_X:\Sigma X^K\xrightarrow{\simeq}\Sigma\widehat{X}^K
\end{equation}
which is natural with respect to $X$, i.e. for a map $f:X\to Y$, the square diagram
$$\xymatrix{\Sigma X^K\ar[r]^\epsilon\ar[d]_{\Sigma f^K}&\Sigma\widehat{X}^K\ar[d]^{\Sigma\hat{f}^K}\\
\Sigma Y^K\ar[r]^\epsilon&\Sigma\widehat{Y}^K}$$
is homotopy commutative, where the vertical arrows are induced from $f$. 

\begin{proposition}
\label{retraction}
If $X$ is a CW-complex and $2(\dim K+1)<m$, the inclusion $j:X^K\to\Delta_K(X)$ has a homotopy retraction after a suspension.
\end{proposition}

\begin{proof}
Let $E:X\to\Omega\Sigma X$ be the suspension map. Since $\Sigma E$ has a retraction, we easily see that the induced map $\Sigma\widehat{E}^K:\Sigma\widehat{X}^K\to\Sigma\widehat{\Omega\Sigma X}{}^K$ has a retraction, say $r$. If $Y$ is an H-space, the map
$$Y\times Y^K\to\Delta_K(Y),\quad(y,(y_1,\ldots,y_m))\mapsto(y\cdot y_1,\ldots,y\cdot y_m)$$
is a map between homotopy fibrations with common base and fiber, and then is a weak homotopy equivalence. Hence if $Y$ has the homotopy type of a CW-complex, the map is a homotopy equivalence, implying that there is a homotopy retraction $r':\Delta_K(Y)\to Y^K$ of the inclusion $j:Y^K\to\Delta_K(Y)$. Combining the above maps, we get a homotopy commutative diagram
$$\xymatrix{\Sigma\widehat{X}^K\ar@{=}[r]\ar@{=}[dd]&\Sigma\widehat{X}^K\ar[r]^{\epsilon^{-1}}\ar[dd]^{\Sigma\widehat{E}^K}&\Sigma X^K\ar[r]^(.45){\Sigma j}\ar[d]^{\Sigma E^K}&\Sigma\Delta_K(X)\ar[d]^{\Sigma\Delta_K(E)}\\
&&\Sigma(\Omega\Sigma X)^K\ar[r]^(.45){\Sigma j}\ar@{=}[d]&\Sigma\Delta_K(\Omega\Sigma X)\ar@{=}[d]\\
\Sigma\widehat{X}^K&\Sigma\widehat{\Omega\Sigma X}{}^K\ar[l]_r&\Sigma(\Omega\Sigma X)^K\ar[l]_\epsilon&\Sigma\Delta_K(\Omega\Sigma X)\ar[l]_{\Sigma r'}}$$
where $\Delta_K(E):\Delta_K(X)\to\Delta_K(\Omega\Sigma X)$ is induced from $E$. Thus the composite
$$\Sigma\Delta_K(X)\xrightarrow{\Sigma\Delta_K(E)}\Sigma\Delta_K(\Omega\Sigma X)\xrightarrow{\Sigma r'}\Sigma(\Omega\Sigma X)^K\xrightarrow{\epsilon}\Sigma\widehat{\Omega\Sigma X}{}^K\xrightarrow{r}\Sigma\widehat{X}^K\xrightarrow{\epsilon^{-1}}\Sigma X^K$$
is the desired homotopy retraction.
\end{proof}

\begin{proof}
[Proof of Theorem \ref{main}]
If $2(\dim K+1)<m$, there is a homotopy fibration $X^K\to\Delta_K(X)\to X$, where the fiber inclusion has a homotopy retraction after a suspension by Proposition \ref{retraction}.
Then by Lemma \ref{decomp}, we get a homotopy equivalence
$$\Sigma\Delta_K(X)\simeq\Sigma X\vee\Sigma X^K\vee\Sigma(X\wedge X^K).$$
Therefore the proof is completed by \eqref{BBCG}.
\end{proof}

\begin{proof}
[Proof of Corollary \ref{euler-char}]
Since $X$ is a compact manifold, $\Delta_K(X)$ is a compact, locally contractible subset of an $mn$-manifold $X^m$. Then by the Poincar\'e-Alexander duality \cite[Proposition 3.46]{H}, there is an isomorphism 
$$H_i(X^m,\M_K(X);\ZZ/2)\cong H^{mn-i}(\Delta_K(X);\ZZ/2),$$
implying that $\chi(X^m,\M_K(X))=(-1)^{mn}\chi(\Delta_K(X))$.
Thus since $\chi(\widehat{X}^k)=(\chi(X)-1)^k+1$ for $k\ge 1$, it follows from Theorem \ref{main} that
$$\chi(X^m,\M_K(X))=(-1)^{mn}\chi(X)(1+\sum_{\emptyset\ne\sigma\in K}(\chi(X)-1)^{|\sigma|}).$$
Therefore the proof is completed by the equality $\chi(X^m)=\chi(X^m,\M_K(X))+\chi(\M_K(X))$.
\end{proof}

\end{document}